\documentclass[11pt]{amsart}

\usepackage{amsmath}
\usepackage{amssymb}

\newtheorem{thm}{Theorem}[section]

\newtheorem{cor}[thm]{Corollary}
\newtheorem{defi}[thm]{Definition}

\begin{document}

\title[Convergence of simultaneous Fourier-Pad\'e approximation]
{Convergence of row sequences of simultaneous Fourier-Pad\'e approximation}

\author[J. Cacoq]{J. Cacoq}
\address{Dpto. de Matem\'aticas\\
Escuela Polit\'ecnica Superior \\
Universidad Carlos III de Madrid \\
Universidad 30, 28911 Legan\'es, Spain} \email{jcacoq@math.uc3m.es}
\thanks{The work of both authors was
supported by MINCINN under grant MTM2009-12740-C03-01}

\author[G. L\'opez ]{G. L\'opez Lagomasino}
\address{Dpto. de Matem\'aticas\\
Escuela Polit\'ecnica Superior \\
Universidad Carlos III de Madrid \\
Universidad 30, 28911 Legan\'es, Spain} \email{lago@math.uc3m.es}



%\dedicatory{Dedicated to A. A. Gonchar, on the occasion of his
%eightieth birthday}
\begin{abstract} We consider row sequences of simultaneous rational approximations constructed in terms of Fourier expansions and prove a Montessus de Ballore type theorem.

\end{abstract}
\date{\today}
\maketitle

%%%%%%%%%%%%%%%%%%%%%%%%%%%%%%%%%%%%%%%%%%%%%%%%%%%%%%%%%%%%%%%%%%%%%%%%%%%%%%%%%%%%%%%%%%%%%%%%%%%%%%%%%%%%%%%%%%%
%%%%%%%%%%%%%%%%%%%%%%%%%%%%%%%%%%%%%%%%%%%%%%%%%%%%%%%%%%%%%%%%%%%%%%%%%%%%%%%%%%%%%%%%%%%%%%%%%%%%%%%%%%%%%%%
%%%%%%%%%%%%%%%%%%%%%%%%%%%%%%%%%%%%%%      INTRODUCTION   %%%%%%%%%%%%%%%%%%%%%%%%%%%%%%%%%%%%%%%%%%%%%%%%%%
%%%%%%%%%%%%%%%%%%%%%%%%%%%%%%%%%%%%%%%%%%%%%%%%%%%%%%%%%%%%%%%%%%%%%%%%%%%%%%%%%%%%%%%%%%%%%%%%%%%%%%%%%%%%%%%%%%
%%%%%%%%%%%%%%%%%%%%%%%%%%%%%%%%%%%%%%%%%%%%%%%%%%%%%%%%%%%%%%%%%%%%%%%%%%%%%%%%%%%%%%%%%%%%%%%%%%%%%%%%%%%%%%%%%%%

\section{Introduction}
Let $\mathbb{T} = \{z:|z|=1\}$ denote the unit circle and $\mathbb{D} = \{z:|z|<1\}$ the open unit disk. By $\sigma$ we denote a finite positive Borel measure whose support is contained in $\mathbb{T}$ and  $\sigma' > 0$ a.e. on $\mathbb{T}$. Let $\{\varphi_n\}$ be the corresponding sequence of orthonormal polynomials with positive leading coefficients
\[ \frac{1}{2\pi}\int \varphi_j(z) \overline{\varphi_k(z)} d\sigma(z) = \delta_{j,k},\qquad j,k \in {\mathbb{Z}}_+,
\]
where as usual $\delta_{j,k} =0, j\neq k$ and $\delta_{k,k} =1$. By ${\mathcal{H}}(\overline{\mathbb{D}})$ we denote the space of functions which are analytic on some neighborhood of $\overline{\mathbb{D}}$.

\begin{defi}\label{defsimultaneosFP}
Let ${\bf f} = (f_1,\ldots,f_d)$ where $f_k \in {\mathcal{H}}(\overline{\mathbb{D}}),k=1,\ldots,d$. Fix a multi-index ${\bf m} =
({m_1},\ldots,m_d) \in {\mathbb{Z}}_+^d \setminus \{\bf 0\}$ where
${\bf 0}$ denotes the zero vector in ${\mathbb{Z}}_+^d$. Set $|{\bf
m}| = {m_1} +\cdots + m_d$. Then, for each $n \geq \max
\{{m_1},\ldots,m_d\}$, there exist polynomials $Q_{n,{\bf m}}, P_{n,{\bf m},j},
j=1,\ldots,d,$ such that
\begin{itemize}
\item[a.1)] $\deg P_{n,{\bf m},j} \leq n - m_j, j=1,\ldots,d,\quad \deg Q_{n,{\bf m}}
\leq |\mathbf{m}|,\quad Q_{n,{\bf m}} \not\equiv 0,$
\item[a.2)] $[Q_{n,{\bf m}} f_j - P_{n,{\bf m},j} ](z) = A_{n,n+1}^{(j)} \varphi_{n+1}(z) + A_{n,n+2}^{(j)} \varphi_{n+2}(z) + \cdots .$
\end{itemize}
We call the vector rational function ${\bf R}_{n,{\bf m}} = (P_{n,{\bf m},1} /Q_{n,{\bf m}}
,\ldots,P_{n,{\bf m},d}/Q_{n,{\bf m}})$  an $(n,{\bf m})$ simultaneous Fourier-Pad\'e
approximation of ${\bf f}$.
\end{defi}

Obviously, the numbers $A^{(j)}_{n,k}$ also depend on  ${\bf m}$ but to simplify the notation we will not indicate it.

It is easy to see that for any pair $(n,{\bf m})$ there is  at least one $R_{n,{\bf m}}$ but,
in general, it is not uniquely determined. In the sequel, we assume that given $(n,{\bf m})$, one solution is taken. We will normalize the common denominator  in terms of its zeros $z_{n,k}$ as follows
\begin{equation} \label{eq:norm}  Q_{n,{\bf m}}(z) = \prod_{|z_{n,k}| \leq 1} (z - z_{n,k})\prod_{|z_{n,k}| > 1}(1 - \frac{z}{z_{n,k}}).
\end{equation}

Obviously, $Q_{n,{\bf m}} f_j - P_{n,{\bf m},j} \in {\mathcal{H}}(\overline{\mathbb{D}})$. Due to the asymptotic properties satisfied by $\{\varphi_n\}$ it is easy to verify that the Fourier expansion of this function converges uniformly on compact subsets of a neighborhood of $\overline{\mathbb{D}}$. (This will be justified later.)

The object of this paper is to prove, under appropriate assumptions on $\bf f$, that
\[ \lim_{n \to \infty} {\bf R}_{n,{\bf m}} = {\bf f}
\]
uniformly on compact subsets of the largest disk centered at $z=0$ containing at most $|{\bf m}|$ poles (in a sense to be described later). Moreover, we will see that under those assumptions the zeros of the common denominator  of the approximating rational functions point out the location and order of the poles of ${\bf f}$ in that disk.

For the case of Taylor expansions of a scalar function $(d=1)$, the corresponding result is called the Montessus de Ballore theorem, see \cite{Mon}. Also in the scalar case, S.P. Suetin gives in \cite{Sue} an extension of Montessus' result to Fourier expansions with respect to an orthonormal system of polynomials with respect to a measure supported on the real line. Another inspiring work  has been a version of Montessus' theorem given by P.R. Graves-Morris and E.B. Saff in \cite{GS1} within the context of vector functions and Taylor expansions (see also \cite{GS2} and \cite{GS3}). In \cite{GS1} the following concept is used.

\begin{defi}\label{defindependenciapolar1}
A vector ${\bf f} = (f_1,\ldots,f_d)$ of functions meromorphic in
some domain $D$ is said to be polewise independent with respect to
the multi-index ${\bf m} = ({m_1},\ldots,m_d) \in {\mathbb{Z}}_+^d
\setminus \{\bf 0\}$ in $D$ if there do not exist polynomials
$p_1,\ldots,p_d$, at least one of which is non-null, satisfying
\begin{itemize}
\item[b.1)] $p_j \equiv 0$ if $m_j =0,$
\item[b.2)] $\deg p_j \leq m_j -1, j=1,\ldots,d,$ if $m_j \geq 1,$
\item[b.3)] $\sum_{j=0}^d p_j f_j \in {\mathcal{H}}(D)$,
\end{itemize}
where ${\mathcal{H}}(D)$ denotes the space of analytic functions in
$D$.
\end{defi}
When $d=1$ polewise
independence merely expresses  that the function has at least $m_1 = {\bf m}$ poles in
$D$. For vector functions we adopt the following definition of a pole and its order.

\begin{defi}\label{defpolo}
Let ${\bf f} = (f_1,\ldots,f_d)$ be a system of analytic functions and
$\mathbf{D}=\left(D_1,\dots,D_d\right)$ a system of domains such
that, for each $k=1,\dots,d,$ $f_k$ is meromorphic in $D_k$. We say
that $\zeta$  is a pole of ${\bf f}$ in $\mathbf{D}$ of order
$\tau$ if there exists an index $k\in\{1,\dots,d\}$  such that $\zeta \in
D_k$ and it is a pole of $f_k$ of order $\tau$, and for the rest of the
indices $j\not = k$ either $\zeta$ is a pole of $f_j$ of order less than
or equal to $\tau$ or $\zeta \not\in D_j$.
\end{defi}

%\begin{defi}\label{defpolos0}
% A point $a$ is said to be a pole of ${\bf f} =
%(f_1,\ldots,f_d)$ of order $\tau$ if it is a pole of that order for
%at least one of the functions $\{f_k\}_{k=1}^d$  and it is of order
%less than or equal to $\tau$ for the rest, including the case that
%it may be a point of analyticity for some of the functions.
%\end{defi}

Polewise independence of $\mathbf{f}$ with respect to $\mathbf{m}$
in $D$  implies that $\mathbf{f}$ has at least $|\mathbf{m}| $ poles
in $\mathbf{D}={(D,\dots,D)}$ counting multiplicities, see  \cite[Lemma 1]{GS1}.
When $\mathbf{D}={(D,\dots,D)}$ we say that  $\zeta$  is a pole of ${\bf f}$ in $D$.

Given ${\bf f} \in {\mathcal{H}}(\overline{\mathbb{D}})$ (that is, each component of ${\bf f}$ is analytic in a neighborhood of $\overline{\mathbb{D}}$) let $D_{|{\bf m}|}({\bf f})$ denote the largest disk centered at the origin inside of which ${\bf f}$ has at most $|{\bf m}|$ poles and $R_{|{\bf m}|}({\bf f})$ denotes its radius.

Let
$Q_{|\bf m|}(\mathbf{f})$ be the polynomial whose zeros are
the poles of ${\bf f}$ in $D_{|{\bf m}|}({\bf f})$ counting multiplicities normalized as in \eqref{eq:norm}. This set of poles is denoted by
${\mathcal{P}}_{|{\bf m}|}({\bf f})$. We prove the following analogue of the Graves-Morris/Saff theorem contained in \cite{GS1} (see also \cite{CCG}).

\begin{thm} \label{corfund}  Assume that
$ {\bf f} \in {\mathcal{H}}(\overline{\mathbb{D}})$ and $\sigma' >0$ a.e. on $\mathbb{T}$. Fix a multi-index ${\bf m}\in {\mathbb{Z}}_+^d
\setminus \{\bf 0\}$ and suppose that ${\bf f}$ is polewise
independent with respect to ${\bf m}$ in $D_{|{\bf m}|}({\bf f})$, Then, ${\bf R}_{n,\bf m}$ is uniquely determined for all sufficiently large $n$. For any compact subset $K$ of $D_{|{\bf m}|}({\bf f})
\setminus {\mathcal{P}}_{|{\bf m}|}({\bf f})$
\begin{equation}\label{inequality3}
 \limsup_{n \to \infty}
\|f_i - R_{n,{\bf m},i}\|_{K}^{1/n} \leq \frac{\| z\|_ {K}}{R_{|{\bf
m}|}({\bf f})}, \quad i=1,\ldots,d,
\end{equation}
where $\|z\|_K$ is replaced by $1$ when $K \subset \overline{\mathbb{D}}$ . Additionally,
\begin{equation}\label{inequality4}
\limsup_{n \to \infty} \| {Q}_{|\mathbf{m}|}(\mathbf{f}) -
Q_{n,{\bf m}}\|^{1/n} \leq \frac{\max\{|\zeta|: \zeta \in
{\mathcal{P}}_{|{\bf m}|}({\bf f})\}}{R_{|{\bf m}|}({\bf f})}.
\end{equation}
\end{thm}

In the space of polynomials of degree $\leq |{\bf m}|$ all norms are equivalent so in \eqref{inequality4} any norm can be taken.

The paper is organized as follows. Section 2 is dedicated to the study of a new construction which we call incomplete Fourier-Pad\'e approximation and prove in some weak sense a Montessus type theorem for such approximants. In the last section we apply the results obtained in Section 2 to the case of simultaneous Fourier-Pad\'e approximation and prove  Theorem \ref{corfund}.

%%%%%%%%%%%%%%%%%%%%%%%%%%%%%%%%%%%%%%%%%%%%%%%%%%%%%%%%%%%%%%%%%%%%%%%%%%%%%%%%%%%%%%%%%%%%%%%%%%%%%%%%%%%%%%%%%%%
%%%%%%%%%%%%%%%%%%%%%%%%%%%%%%%%%%%%%%%%%%%%%%%%%%%%%%%%%%%%%%%%%%%%%%%%%%%%%%%%%%%%%%%%%%%%%%%%%%%%%%%%%%%%%%%
%%%%%%%%%%%%%%%%%%%%%%%%%%%%%%%%%%%%%%      Incomplete   %%%%%%%%%%%%%%%%%%%%%%%%%%%%%%%%%%%%%%%%%%%%%%%%%%
%%%%%%%%%%%%%%%%%%%%%%%%%%%%%%%%%%%%%%%%%%%%%%%%%%%%%%%%%%%%%%%%%%%%%%%%%%%%%%%%%%%%%%%%%%%%%%%%%%%%%%%%%%%%%%%%%%
%%%%%%%%%%%%%%%%%%%%%%%%%%%%%%%%%%%%%%%%%%%%%%%%%%%%%%%%%%%%%%%%%%%%%%%%%%%%%%%%%%%%%%%%%%%%%%%%%%%%%%%%%%%%%%%%%%%

\section{Incomplete Fourier-Pad\'e approximants}\label{incompletos}

Let us introduce the type of convergence which is relevant in this section.
Let $B$ be a subset of the complex
plane $\mathbb{C}$. By $\mathcal{U}(B)$ we denote the class of all
coverings of $B$ by at most a numerable set of disks. Set
$$
h(B)=\inf\left\{\sum_{i=1}^\infty
|U_i|\,:\,\{U_i\}\in\mathcal{U}(B)\right\},
$$
where $|U_i|$ stands for the radius of the disk $U_i$. The quantity
$h(B)$ is called the $1$-dimensional Hausdorff content of the
set $B$. This set function is not a measure but it is  semi-additive
and monotonic, properties which will be used later. Clearly, if $B$
is a disk then $h(B)=|B|$.
\begin{defi}\label{defcontenido}
Let $\{g_n\}_{n\in\mathbb{N}}$ be a sequence of functions
defined on a domain $D\subset\mathbb{C}$ and $g$ another
function defined on $D$. We say that
$\{g_n\}_{n\in\mathbb{N}}$ converges in $h$-content to
the function $g$ on compact subsets of $D$ if for every compact
subset $K$ of $D$ and for each $\varepsilon
>0$, we have
$$
\lim_{n\to\infty} h\{z\in K :
|g_n(z)-g(z)|>\varepsilon\}=0.
$$
Such a convergence will be denoted by $h$-$\lim_{n\to\infty}
g_n = g$ in $D$.
\end{defi}

Now, we introduce a new type of approximants which will be useful for  achieving our main goal.

\begin{defi}\label{defincompletosFP}
Let $f \in {\mathcal{H}}(\overline{\mathbb{D}})$.   Fix
${m^*} \leq m$. Let $n \geq m$. We say that the rational function
$R_{n,m} = P_{n,m}/Q_{n,m}$ is an incomplete Fourier-Pad\'e approximation of type $(n,m,{m^*})
$ corresponding to $f$ if   $P_{n,m} $ and $Q_{n,m} $ are polynomials that verify
\begin{itemize}
\item[c.1)]
$\deg P_{n,m} \le n-{m^*},\quad \deg Q_{n,m} \le m,\quad Q_{n,m} \not\equiv 0,$
\item[c.2)] $[Q_{n,m} f-P_{n,m}](z)=
a_{n,n+1} \varphi_{n+1}(z) + a_{n,n+2} \varphi_{n+2} + \cdots .$
\end{itemize}
\end{defi}

The polynomials $P_{n,m},Q_{n,m}$ depend on $m^*$ and the numbers $a_{n,k}$ depend on $(m,m^*)$ but we do not indicate it to reduce the notation.

From Definitions
\ref{defsimultaneosFP} and \ref{defincompletosFP} it follows that
$R_{n,{\bf m},k}, k=1,\ldots,d,$  is an incomplete Pad\'e
approximation of type $(n,|\mathbf{m}|,m_k)$ with respect to $f_k$.

Given $n \geq m \geq {m^*},$ $R_{n,m}$ is not unique so  we choose one candidate. The denominator $Q_{n,m}$ will be normalized as in \eqref{eq:norm}.

Our purpose in this section is to prove that  $h-\lim_{n\to \infty}R_{n,m} = f$ in  the largest disk $D_{m^*}(f)$ inside of which $f$ has at most $m^*$ poles.

Take an arbitrary $\varepsilon > 0$ and define the open set
$J_\varepsilon$ as follows. For $n \geq m$, let $J_{n,\varepsilon}$
denote the $\varepsilon/6mn^2$-neighborhood of the set
${\mathcal{P}}_{n,m} = \{\zeta_{n,1}, \ldots, \zeta_{n,m_n}\}$ of
finite zeros of $Q_{n,m}$ and let
$J_{m-1,\varepsilon}$ denote the $\varepsilon/6m$-neighborhood of
the set of poles of $f$ in $D_{m}(f)$. Set $J_\varepsilon = \cup_{n \ge m
-1} J_{n,\varepsilon}$. From monotonicity and subadditivity it is easy to check that $h(J_{\varepsilon}) <
\varepsilon$ and $J_{\varepsilon_1} \subset J_{\varepsilon_2}$ for $
\varepsilon_1 < \varepsilon_2$. For any set $B \subset {\mathbb{C}}$
we put ${B}(\varepsilon) := {B} \setminus J_{\varepsilon} $.

Obviously, if $\{g_n\}_{n \in \mathbb{N}}$ converges
uniformly to $g$ on ${K}(\varepsilon)$ for every compact ${K}
\subset D$ and  $\varepsilon >0$, then
$h$-$\lim_{n\to\infty} g_n = g$ in $D$.

Due to the normalization \eqref{eq:norm}, for any compact set $K$ of
$\mathbb{C}$ and for every $\varepsilon >0$, there exist positive constants
$C_1, C_2$, independent of $n$, such that
\begin{equation} \label{desig1} \|Q_{n,m}\|_{K} < C_1,
\qquad \min_{z \in {K}(\varepsilon)}|Q_{n,m}(z)| > C_2 n^{-2m},
\end{equation}
where the second inequality is meaningful when ${K}(\varepsilon)$ is
a non-empty set.

In the sequel, $C_k$ will be used to denote positive constants, generally
different, that are independent of $n$ but may depend on all
the other parameters involved in each formula where they appear.

Given $\sigma$ and $\varphi_n$ denote
\[ \psi_n(z) = \frac{1}{2\pi} \int \frac{\overline{\varphi_n(\zeta)}d\sigma(\zeta)}{z - \zeta}
\]
If $\sigma' > 0$ a.e. on $\mathbb{T}$, E.A. Rakhmanov's theorem (see, for example, \cite{Rak}) implies that
\begin{equation} \label{Rak1}
\lim_{n \to \infty} \frac{\varphi_{n+k}(z)}{\varphi_n(z)} = z^k,\qquad k \in \mathbb{Z},
\end{equation}
\begin{equation} \label{Rak2}
\lim_{n \to \infty} \frac{\psi_{n+k}(z)}{\psi_n(z)} = \frac{1}{z^k},\qquad k \in \mathbb{Z},
\end{equation}
uniformly on each compact subset of $\mathbb{C} \setminus \overline{\mathbb{D}}$. In turn, these relations easily imply that
\begin{equation} \label{Rak3}
\lim_{n \to \infty} |\varphi_n(z)|^{1/n} = |z|,
\end{equation}
\begin{equation} \label{Rak4}
\lim_{n \to \infty} | {\psi_n(z)}|^{1/n} = |z|^{-1},
\end{equation}
uniformly on each compact subset of $\mathbb{C} \setminus \overline{\mathbb{D}}$. Formulas \eqref{Rak1}-\eqref{Rak4} are basic in our proofs.

Let $g\in \mathcal{H}(\overline{\mathbb{D}})$. Take $r >1$ so that $g\in \mathcal{H}(\{z:|z| \leq r\})$. Set $T_{r} = \{z:|z| = r\}$. Using Cauchy's integral formula and Fubini's theorem, we obtain
\[ \langle g, \varphi_k \rangle = \frac{1}{2\pi} \int g(\zeta) \overline{\varphi_k(\zeta)} d\sigma(\zeta) = \frac{1}{2\pi} \int \frac{1}{2\pi i} \int_{T_r } \frac{g(z)dz}{z-\zeta} \overline{\varphi_k(\zeta)} d\sigma(\zeta) =
\]
\begin{equation} \label{coefF} \frac{1}{2\pi i} \int_{T_r} g(z) \psi_k(z) d z.
\end{equation}
Using \eqref{Rak2}-\eqref{coefF}, it readily follows that
\[ \sum_{k=0}^\infty \langle g, \varphi_k \rangle \varphi_k(z)
\]
converges uniformly on each compact subset of $D_r = \{z: |z| < r\}$ and the limit must be an analytic function in $D_{r}$. On the other hand, this is the Fourier series of $g$ with respect to the orthonormal system $\{\varphi_n\}$; consequently, its norm-2 limit on $\mathbb{T}$ is $g$. By the principle of analytic continuation, the series converges uniformly to $g$ on compact subsets of $D_{r}$. This justifies that the right hand sides of a.2) and c.2) converge uniformly to the corresponding left hand on each compact subset of a neighborhood of $\overline{\mathbb{D}}$.

\begin{thm}\label{contenido}
Let $f \in {\mathcal{H}}(\overline{\mathbb{D}})$ and $\sigma' > 0$ a.e. on $\mathbb{T}$. Fix $m$ and ${m^*}$ nonnegative integers, $m \geq
{m^*}$. For each $n \geq m$, let $R_{n,m}$ be an incomplete Pad\'e
approximant of type $(n,m,{m^*})$ for $f$. Then, for each $\varepsilon > 0$ and every compact subset $K$ of  $D_{m^*}(f)$
\begin{equation} \label{fund4}
\limsup_{n \to \infty}\|f - R_{n,m}\|_{K(\varepsilon)}^{1/n} \leq \frac{\|z\|_K}{R_{m^*}(f)},
\end{equation}
where $\|z\|_K$ should be replaced by $1$ when $K \subset \overline{\mathbb{D}}$. In particular,
$$
h\mbox{-}\lim_{n\to\infty} R_{n,m}  = f \;\; \mbox{in}\;\;
D_{m^*}(f).
$$
Finally, for each pole $z_j$ of $f$ in $D_{m^*}(f)$, and every $\varepsilon > 0$, there exists $n_0$ such that for all $n \geq n_0$ the polynomials $Q_{n,m}$ have at least $\tau_j$ zeros in the disk $\{z:|z - z_j| < \varepsilon\}$, where $\tau_j$ denotes the order of the pole $z_j$.
\end{thm}

\begin{proof} Let $z_1,\ldots,z_N$ be the distinct poles of $f$ in $D_{m^*}(f)$
and $\tau_1,\ldots,\tau_N$ their orders, respectively. Consequently, $\sum_{k=1}^N \tau_k = \widetilde{m} \leq m^*$. Put
\[ w(z) = \prod_{k=1}^N (z - z_k)^{\tau_k}.
\]
Using c.2) we obtain
\begin{equation}\label{ipad}
(w_mQ_{n,m}f-w_mP_{n,m})(z)=\sum_{k \geq n+1}a_{n,k}w_m(z)\varphi_{k}(z) = \sum_{\nu \geq 0} b_{n,\nu} \varphi_\nu(z).
\end{equation}
Notice that $w_mQ_{n,m}f-w_mP_{n,m}\in H(D_{m^{*}}(f))$ and $\deg w_mP_{n,m} \leq n$.

We have two ways of calculating the Fourier coefficients $b_{n,\nu}$. On one hand, for each $\nu \geq 0$
\begin{equation} \label{F1}
b_{n,\nu}= \displaystyle \sum_{k =n+1}^{\infty} a_{n,k} \left < w_m \varphi_k,\varphi_{\nu}
\right > =  \sum_{k =n+1}^{\infty} \frac{a_{n,k}}{2\pi}\int w_m(z)\varphi_{k}(z)\overline{\varphi_{\nu}(z)}d \sigma(z).
\end{equation}
On the other hand
$$ b_{n,\nu}=\left\{ \begin{array}{ll}
\langle w_mQ_{n,m}f-w_mP_{n,m^{*}}, \varphi_{\nu}  \rangle,  & \nu=0,\cdots,n, \\
\langle  w_mQ_{n,m}f, \varphi_{\nu}  \rangle, &  \nu \geq n+1 .
\end{array}
\right.  $$
Since  $w_mQ_{n,m}f$ is analytic in $D_{m^{*}}(f)$ taking $1 < R<R_{m^{*}}(f)$, we obtain (see \eqref{coefF})
\begin{equation} \label{F2}
b_{n,\nu}= \frac{1}{2 \pi i}\int_{T_R} (w_mQ_{n,m}f-w_mP_{n,m })(z)  {\psi_{\nu}(z)} dz,  \qquad \nu=0,...,n,
\end{equation}
and, similarly,
\begin{equation} \label{F3}
b_{n,\nu}= \frac{1}{2 \pi i}\int_{T_R} (w_mQ_{n,m}f)(z)  {\psi_{\nu}(z)} dz.\qquad \nu \geq n+1.
\end{equation}

We will show that $\sum_{\nu \geq 0} b_{n,\nu}\varphi_\nu (z)$ converges uniformly to zero on each compact subset of $D_{m^*}(f)$ as $n\to \infty$ with geometric rate. To this end, due to the maximum principle, without loss of generality we assume that the compact sets contain the closed unit disk. Let us separate the series in two
\begin{equation} \label{resto} \sum_{\nu \geq 0} b_{n,\nu}\varphi_\nu(z) = \sum_{\nu = 0}^{n} b_{n,\nu}\varphi_\nu(z) + \sum_{\nu \geq n+1} b_{n,\nu}\varphi_\nu(z)
\end{equation}
Fix a compact subset $K, \overline{\mathbb{D}} \subset K  \subset D_{m^*}(f)$.

We begin with the infinite series to the right of \eqref{resto} which is easier to handle. Take $1 < R < R_{m^*}(f)$ such that $K$ and the poles of $f$ in $D_{m^*}(f)$ are  surrounded by $T_R$. According to \eqref{F3} and the first inequality in \eqref{desig1}
\[ |b_{n,\nu}| \leq C\|\psi_\nu\|_{T_R},
\]
for some constant $C$ independent of $n$. Choose $\delta > 0$ sufficiently small so that $\|z\|_{K} + \delta < R - \delta$. According to \eqref{Rak3}-\eqref{Rak4}, there exists $n_0$ such that
\[ \|\psi_\nu\|_{T_R} \leq \frac{1}{(R-\delta)^\nu}, \quad \|\varphi_\nu\|_{K} \leq (\|z\|_K + \delta)^\nu, \qquad \nu \geq n_0
\]
Set $q = \frac{\|z\|_K + \delta}{R-\delta} (< 1)$. If $n \geq n_0$, we obtain
\[\sum_{\nu \geq n+1} |b_{n,\nu}|\|\varphi_\nu(z)\|_K \leq C \sum_{\nu \geq n+1} q^\nu = C \frac{q^{n+1}}{1 - q}.
\]
Taking $\limsup_n$ of the $n$th root, then making $\delta$ tend to zero and $R$ to $R_{m^*}(f)$ it readily follows that
\begin{equation} \label{fund1} \limsup_n  \|\sum_{\nu \geq n+1} b_{n,\nu} \varphi_\nu \|_K^{1/n} \leq \frac{\|z\|_K}{R_{m^*}(f)}.
\end{equation}
Should $K$ be contained in $\overline{\mathbb{D}}$ one  substitutes $\|z\|_{K}$ by $1$ in the formula.

To deal with the first sum, we begin estimating the values $a_{n,k}$. The trick we are about to exhibit was borrowed from \cite{Sue}. Choose $r > 1$ such that $T_r \subset D_0(f)$ and $r < R < R_{m^*}(f)$ such that all the poles of $f$ in $D_{m^*}(f)$ and the compact set $K$ are surrounded by $T_R$. Set
$$a_{n,k}=\frac{1}{2 \pi i}\int_{T_r}(Q_{n,m}f\psi_k)(z)dz, \qquad k \geq n-m^*+1,  $$
$$\gamma_{n,k}=\frac{1}{2 \pi i}\int_{T_R}(Q_{n,m}f \psi_k)(z)dz, \qquad k \geq n-m^*+1. $$
Notice that $a_{n,k} = 0, k=n-m^*+1,\ldots,n,$ and for $k \geq n+1$ the $a_{n,k}$ are precisely the Fourier coefficients on the right hand side of c.2) (see \eqref{coefF}).

Since $f$ is meromorphic in  $D_{m^{*}}(f)$, using the residue theorem we have
\begin{equation}\label{ipad2}
\gamma_{n,k}-a_{n,k}=\displaystyle \sum_{j=1}^{N}\displaystyle \mbox{Res}(Q_{n,m}f\psi_k,z_j), \quad k\geq n-m^* +1,
 \end{equation}
where $\mbox{Res}(Q_{n,m}f\psi_k(z),z_j)$ is the residue of $Q_{n,m}f\psi_k$  at $z_j$. At $z_j$ the function $Q_{n,m}f\psi_k$ has a pole of order  $\leq \tau_j$; therefore,
 $$\mbox{Res}  (Q_{n,m}f\psi_k,z_j)=$$
 $$\frac{1}{(\tau_j-1)!}\lim_{z \to  z_j}\left[(Q_{n,m}\psi_{n})(z)\frac{(z-z_j)^{\tau_j}f(z)\psi_k(z)}{\psi_{n}(z)}   \right]^{(\tau_j-1)}.$$
Using the Leibnitz formula, it follows that
\[
\left[(Q_{n,m}\psi_{n})(z)\frac{(z-z_j)^{\tau_j}f(z)\psi_k(z)}{\psi_{n}(z)}   \right]^{(\tau_j-1)}=
\]
\begin{equation}\label{ipad4}
\sum_{\ell=0}^{\tau_j-1} {\tau_j-1 \choose \ell} \left[ (Q_{n,m}\psi_{n })(z) \right]^{(\tau_j -1-\ell)} \left[ \frac{(z-z_j)^{\tau_j}f(z)\psi_k(z)}{\psi_{n }(z)}\right]^{(\ell)}.
\end{equation}

Define
\begin{equation}\label{ipad5}
\alpha_{n}(j,\ell):=\frac{1}{(\tau_j-1)!}{\tau_j -1 \choose \ell}\lim_{z \to z_j}( Q_{n,m} \psi_{n })^{(\tau_j-1-\ell)}(z) .
\end{equation}
By \eqref{ipad4} and \eqref{ipad5}, we obtain
\begin{equation} \label{ipad3}
\mbox{Res}(Q_{n,m}f\psi_k,z_j)=\sum_{\ell=0}^{\tau_j-1}\alpha_{n}(j,\ell)\left[ \frac{(z-z_j)^{\tau_j}f(z)\psi_k(z)}{\psi_{n }(z)}\right]^{(\ell)}_{z=z_j}.
\end{equation}
Notice that $\alpha_n(j,\ell)$ does not depend on $k$. Then, for each $k \geq n-m^* +1$, \eqref{ipad2} and  \eqref{ipad3} give
\begin{equation} \label{as} a_{n,k}=\gamma_{n,k}- \displaystyle \sum_{j=1}^{N} \sum_{\ell=0}^{\tau_j-1}\alpha_{n}(j,\ell)\left[ \frac{(z-z_j)^{\tau_j}f(z)\psi_k(z)}{\psi_{n }(z)}\right]^{(\ell)}_{z=z_j}.
 \end{equation}
Since $a_{n,k}=0$, for $k=n-m^{*}+1,\ldots,n$ we can write
\begin{equation}\label{irod}
\gamma_{n,k}=   \displaystyle \sum_{j=1}^{N} \sum_{\ell=0}^{\tau_j-1}\alpha_{n}(j,\ell)\left[ \frac{(z-z_j)^{\tau_j}f(z)\psi_k(z)}{\psi_{n }(z)}\right]^{(\ell)}_{z=z_j}.
\end{equation}
Recall that $\sum_{j=1}^d \tau_j = \widetilde{m} \leq m^*$. Thus we have obtained a system of $\widetilde{m}$ equations on $\widetilde{m}$ unknowns (the quantities $\alpha_n(j,\ell)$).

The determinant $\Delta_n$ of the system has the form
\[ \Delta_n =  \left|
\begin{smallmatrix}
\left[\frac{(z-z_j)^{\tau_j}f(z)\psi_{n-\widetilde{m}+1}(z)}{\psi_{n}(z)}\right]_{z=z_j}  &  \cdots & \left[\frac{(z-z_j)^{\tau_j}f(z)\psi_{n-\widetilde{m}+1}(z)}{\psi_{n}(z)}\right]_{z=z_j}^{(\tau_j-1)}  \\
\left[\frac{(z-z_j)^{\tau_j}f(z)\psi_{n-\widetilde{m}+2}(z)}{\psi_{n}(z)}\right]_{z=z_j} &  \cdots & \left[\frac{(z-z_j)^{\tau_j}f(z)\psi_{n-\widetilde{m}+2}(z)}{\psi_{n}(z)}\right]_{z=z_j}^{(\tau_j-1)} \\
\vdots & \vdots & \vdots \\
\left[\frac{(z-z_j)^{\tau_j}f(z)\psi_{n}(z)}{\psi_{n}(z)}\right]_{z=z_j} &  \cdots & \left[\frac{(z-z_j)^{\tau_j}f(z)\psi_{n}(z)}{\psi_{n}(z)}\right]_{z=z_j}^{(\tau_j-1)}
\end{smallmatrix}
\right|_{j=1,\ldots,N} ,
\]
where the subindex on the determinant means that the indicated group of columns are successively written for $j=1,2,\ldots,N$. Due to \eqref{Rak3}
$$ \lim_{n \longrightarrow \infty} \Delta_{n}= {\Delta},$$
where
$$ {\Delta}=
\left|\begin{smallmatrix}
[(z-z_j)^{\tau_j}z^{\widetilde{m}-1}f(z)]_{z=z_j} &   \cdots & [(z-z_j)^{\tau_j}z^{\widetilde{m}-1}f(z)]_{z=z_j}^{(\tau_j-1)}(z_j)  \\
[(z-z_j)^{\tau_j}z^{\widetilde{m}-2}f(z)]_{z=z_j}&  \cdots & [(z-z_j)^{\tau_j}z^{\widetilde{m}-2}f(z)]_{z=z_j}^{(\tau_j-1)}(z_j)\\
\vdots & \vdots & \ddots & \vdots \\
 [(z-z_j)^{\tau_j}f(z)]_{z=z_j}  &  \cdots & [(z-z_j)^{\tau_j}f(z)]_{z=z_j} ^{(\tau_j-1)}.
\end{smallmatrix}\right|_{j=1,2,\cdots,N}.
 $$

Notice that $\Delta \neq 0$. In fact, should this determinant be equal to zero that would mean that there exists a linear combination of its rows giving the zero vector. In turn, this implies that there exists a polynomial of degree $\leq \widetilde{m} -1$ which multiplied times $f$ eliminates the $\widetilde{m}$ poles which $f$ has in $D_{m^*}(f)$ which is clearly impossible.  Therefore, $|\Delta_n| \geq C > 0$ for all sufficiently large $n$. In the sequel we only consider such $n$'s.

Let $\Delta_n(j,\ell)$ denote the determinant which is obtained substituting in the determinant of the system the column with index $q = \sum_{i=1}^{j-1} \tau_i + \ell +1$ with the column vector $(\gamma_{n,n-\widetilde{m}+1},\ldots,\gamma_{n,n})^{t}$ formed with the independent terms of equations \eqref{irod}. By Cramer's rule
\begin{equation}\label{ipad7}
\alpha_{n}(j,\ell)=\frac{\Delta_{n}(j,\ell)}{\Delta_n} = \frac{1}{\Delta_n} \sum_{s=1}^{\widetilde{m}} \gamma_{n,n-\widetilde{m} +s} M_n(s,q).
\end{equation}
where $M_n(s,q)$ is the cofactor corresponding to row $s$ and column $q$ of $\Delta_n(j,\ell)$. Making use of the fact that the $\alpha_n(j,\ell)$ do not depend on $k$ from \eqref{as} and \eqref{ipad7} it follows that
\begin{equation}\label{eqx}
a_{n,k}=\gamma_{n,k}-
\end{equation}
\[\frac{1}{\Delta_n}\sum_{j=1}^{N}  \sum_{\ell=0}^{\tau_j-1}\sum_{s=1}^{m^{*}}\gamma_{n,n-\widetilde{m}+s}M_n(s,q) \left(\frac{\psi_k}{\psi_{n }}  \right)^{(\ell)}(z_j)  , \quad  k\geq n-m^* +1.
\]

Choose $\varepsilon > 0$ so that $|z_j| - \varepsilon > r$ for all $j=1,\ldots,N$. (Recall that $r$ was chosen greater than $1$.) Using Cauchy's integral formula
\[ \left(\frac{\psi_{k}}{\psi_{n }}\right)^{(\ell)}(z_j) = \frac{\ell!}{2\pi i} \int_{|z-z_j| = \varepsilon}\frac{\psi_k(z)d z}{\psi_{n }(z)(z-z_j)^{\ell+1}}.
\]
On account of \eqref{Rak2}, it follows that there exists a constant $C_1$ such that
\[ \left| \left(\frac{\psi_{k}}{\psi_{n }}\right)^{(\ell)}(z_j)\right| \leq C_1\frac{1}{r^{k-n }}, \qquad k \geq n-m^* +1
\]
for all $ j=1\ldots,N, l=0,1,\cdots,\tau_j -1,$ and $n$ sufficiently large.
Consequently,
$$ \left | M_n(s,q) \right| \leq C_2$$
and using  \eqref{eqx} we obtain that there exists a constant $C_3$ such that
\begin{equation} \label{estas}
|a_{n,k}|\leq |\gamma_{n,k}|+\frac{C_3}{r^{k-n }}\sum_{s=1}^{m^{*}}|\gamma_{n-\widetilde{m}+s}|, \qquad k\geq n+1.
\end{equation}

From the integral which defines $\gamma_{n,k}$, the first inequality in \eqref{desig1}, and using \eqref{Rak3} we have that given $\delta > 0, R-\delta > r$, for all sufficiently large $n$
\[ |\gamma_{n,k}| \leq \frac{1}{(R - \delta)^k}, \qquad k \geq n - m^*+1
\]
and taking into consideration \eqref{estas}, we obtain
\begin{equation} \label{eq:as}
|a_{n,k}|\leq   \frac{C_4}{r^{k-n}(R - \delta)^n}, \quad k \geq n+1,
\end{equation}
for some constant $C_4$. Since  $\left|\langle w_m\varphi_{k},\varphi_{\nu}  \rangle \right| \leq \|w_m\|_{\mathbb{T}}$, due to \eqref{F1}  we can find a constant $C_5$ for which
\begin{equation} \label{estas1}|b_{n,\nu}|\leq \frac{C_5}{(R - \delta)^n}.
\end{equation}

Finally, let us estimate $\sum_{\nu = 0}^{n} b_{n,\nu}\varphi_\nu(z)$. Fix a compact subset $K \subset D_{m^*}(f)$. As we did for the other sum, we can assume without loss of generality that $K \supset \overline{\mathbb{D}}$. We also assume that $R$ and $\delta$ chosen previously satisfy $\|z\|_K + \delta < R - \delta$. From \eqref{Rak3} it follows that there exists some constant $C_6$ such that
\[ \|\varphi_\nu\|_K \leq C_6 (\|z\|_K + \delta)^\nu, \qquad \nu \geq 0.
\]
Therefore, making use of \eqref{estas1}, we obtain
\[ \|\sum_{\nu = 0}^{n} b_{n,\nu}\varphi_\nu(z)\|_K \leq  \frac{C_5C_6}{(R - \delta)^n}\sum_{\nu=0}^n (\|z\|_K + \delta)^\nu =
\]
\[\frac{C_5C_6}{(R - \delta)^n} \frac{(\|z\|_K + \delta)^{n+1} -1}{\|z\|_K + \delta -1}.
\]
Taking $\limsup_n$ of the $n$th root, then making $R$ tend to $R_{m^*}(f)$ and $\delta$ to zero, we arrive at
\begin{equation} \label{fund2}
\limsup_{n \to \infty}\|\sum_{\nu = 0}^{n} b_{n,\nu}\varphi_\nu(z)\|_K^{1/n} \leq \frac{\|z\|_K}{R_{m^*}(f)}.
\end{equation}
Again, if $K$ is contained in $\overline{\mathbb{D}}$ one must write $1$ on the right hand of this formula in place of $\|z\|_K$.
Formulas \eqref{ipad} and \eqref{resto} together with inequalities \eqref{fund1} and \eqref{fund2} easily render
\begin{equation} \label{fund3}
\limsup_{n \to \infty}\|w_m(Q_{n,m}f - P_{n,m})\|_K^{1/n} \leq \frac{\|z\|_K}{R_{m^*}(f)}.
\end{equation}

Fix $\varepsilon > 0$ and take any compact subset $K \subset D_{m^*}(f)$. For $z \in K(\varepsilon) = K \setminus J_{\varepsilon}$, according to the second inequality in \eqref{desig1}, we have that (notice that $J(\varepsilon)$ leaves out an $\varepsilon/6m$ neighborhood of the zeros on $w_m$)
\[ \|f - R_{n,m}\|_{K(\varepsilon)} \leq \frac{n^{2m}}{C_7}\|w_m(Q_{n,m}f - P_{n,m})\|_K
\]
for some constant $C_7$, and applying \eqref{fund3}, we obtain \eqref{fund4}. As mentioned in the introduction of the sets $J(\varepsilon)$, \eqref{fund4} (and much less) implies convergence in Hausdorff content in $D_{m^*}(f)$ as claimed. The statement concerning the asymptotic behavior of some of the zeros of $Q_{n,m}$ is a direct consequence of the convergence in Hausdorff content and a lemma of A.A. Gonchar, see \cite[Lemma 1]{gon}. With this we conclude the proof.
\end{proof}

We wish to point out that from \eqref{fund3} one can  derive that the $\tau_j$ poles of $Q_{n,m}$ closest to $z_j$ in fact converge to $z_j$ with geometric rate not greater than $|z_j|/R_{m^*}(f)$. We will return to this later.

We also mention that apart from the application of incomplete Fourier-Pad\'e approximation to simultaneous Fourier-Pad\'e approximation, which  will be seen in the next section, there are other possibilities. For example, we may have apriori knowledge of the location of some of the poles of $f$ and we can use this information to fix some of the zeros of $Q_{n,m}$ at such points. Another possibility is to combine two (or more) approximation criteria to define the rational functions; for example, interpolation at a fixed number of points and the $L_2$ construction (or viceversa as in \cite{S} and te references therein).

%%%%%%%%%%%%%%%%%%%%%%%%%%%%%%%%%%%%%%%%%%%%%%%%%%%%%%%%%%%%%%%%%%%%%%%%%%%%%%%%%%%%%%%%%%%%%%%%%%%%%%%%%%%%%%%%%%%
%%%%%%%%%%%%%%%%%%%%%%%%%%%%%%%%%%%%%%%%%%%%%%%%%%%%%%%%%%%%%%%%%%%%%%%%%%%%%%%%%%%%%%%%%%%%%%%%%%%%%%%%%%%%%%%
%%%%%%%%%%%%%%%%%%%%%%%%%%%%%%%%%%%%%%      MONTESSUS   %%%%%%%%%%%%%%%%%%%%%%%%%%%%%%%%%%%%%%%%%%%%%%%%%%
%%%%%%%%%%%%%%%%%%%%%%%%%%%%%%%%%%%%%%%%%%%%%%%%%%%%%%%%%%%%%%%%%%%%%%%%%%%%%%%%%%%%%%%%%%%%%%%%%%%%%%%%%%%%%%%%%%
%%%%%%%%%%%%%%%%%%%%%%%%%%%%%%%%%%%%%%%%%%%%%%%%%%%%%%%%%%%%%%%%%%%%%%%%%%%%%%%%%%%%%%%%%%%%%%%%%%%%%%%%%%%%%%%%%%%

\section{Simultaneous approximation}\label{simultaneos}
As in the introduction, in this section we have $\mathbf{f}=(f_1,\dots,f_d) \in {\mathcal{H}}(\overline{\mathbb{D}})$
and $\mathbf{m}=(m_1,\dots,m_d)\in\mathbb{Z}^d_+\setminus\{\mathbf{0}\}$.
We will study the convergence of ${R}_{n,{\bf m}}$ to $\bf f$. Recall that for each $k=1,\ldots,d$ the rational function $R_{n,{\bf m},k}$ is an $(n,|{\bf m},m_k)$ incomplete Fourier-Pad\'e approximation to $f_k$.  Let
$\mathcal{P}_{n,{\bf m}}$ be the collection of zeros  of $Q_{n,\mathbf{m}}$. A direct consequence of Theorem \ref{contenido} is the following corollary.

\begin{cor}\label{contenidovec}
Let ${\bf f} \in {\mathcal{H}}(\overline{\mathbb{D}})$ and $\sigma' > 0$ a.e. on $\mathbb{T}$. Fix ${\bf m} \in {\mathbb{Z}}^d \setminus {\bf 0}$.  For each $n \geq |{\bf m}|$, let ${\bf R}_{n,{\bf m}}$ be a Fourier-Pad\'e
approximant of type $(n,{\bf m})$ for ${\bf f}$. Then, for each $i = 1,\ldots,d, K \subset D_{m_k}(f_k)$, and $\varepsilon > 0$
\begin{equation} \label{fund4*}
\limsup_{n \to \infty}\|f_i - R_{n,{\bf m},i}\|_{K(\varepsilon)}^{1/n} \leq \frac{\|z\|_K}{R_{m_i}(f_i)},
\end{equation}
where $\|z\|_K$ should be replaced by $1$ when $K \subset \overline{\mathbb{D}}$. In particular,
$$
h\mbox{-}\lim_{n\to\infty} R_{n,{\bf m},i}  = f_i \;\; \mbox{in}\;\;
D_{m_i}(f_i).
$$
Finally, for each $ i=1,\ldots,d,$ and pole $z_j$ of $f_i$ in $D_{m_i}(f_i)$, for any $\varepsilon > 0$, there exists $n_0$ such that for all $n \geq n_0$ the polynomials $Q_{n,{\bf m}}$ have at least $\tau_j$ zeros in  $\{z:|z - z_j| < \varepsilon\}$, where $\tau_j$ denotes the order of the pole $z_j$.
\end{cor}

Now let us prove Theorem \ref{corfund}. We will combine arguments used to prove Theorem \ref{contenido} and ideas from \cite{GS1}.

\vspace{0,2cm} \noindent
{\bf Proof of Theorem \ref{corfund}.}  First of all, notice that $\bf f$ must have exactly $|\bf m|$ poles in $D_{|{\bf m}|}({\bf f})$. If this was not the case, it is easy to show that $\bf f$ is not polewise independent with respect to $\bf m$ in $D_{|{\bf m}|}({\bf f})$. By $z_1,\ldots,z_N$ we denote the distinct poles of $\bf f$ in $D_{|\bf m|}({\bf f})$ and let $\tau_1,\ldots,\tau_N$ be their orders, respectively.

According to Definition \ref{defsimultaneosFP}, for each $i=1,\ldots,d$ we have
\[ (Q_{n,{\bf m}}f_i - P_{n,{\bf m},i})(z) = A_{n,n+1}^{(i)} \varphi_{n+1}(z) + \cdots,
\]
and $\deg P_{n,{\bf m},i} \leq n -m_i$. Therefore,
\[ A_{n,k} = \langle Q_{n,{\bf m}}f_i,\varphi_{k} \rangle, \qquad k \geq n - m_i +1,
\]
and $A_{n,k}^{(i)} =0, k = n-m_i+1,\ldots,n$. Our first goal will be to estimate the values $A_{n,k}^{(i)}$. The procedure is similar to the one employed to estimate the quantities $a_{n,k}$ in the proof of Theorem \ref{contenido}, so we will not go through all the details, but there are some important aspects to single out.

Take $r > 1$ such that $T_r \subset D_{0}({\bf f})$ and $R < R_{|{\bf m}|}({\bf f})$ such that $T_R$ surrounds all the poles $z_1,\ldots,z_N$. Obviously
$$A_{n,k}^{(i)}=\frac{1}{2 \pi i}\int_{T_r}(Q_{n,{\bf m}}f\psi_k)(z)dz, \qquad k \geq n-m_i+1.  $$
Define
$$\gamma_{n,k}^{(i)}=\frac{1}{2 \pi i}\int_{T_R}(Q_{n,{\bf m}}f \psi_{k})(z)dz, \qquad k \geq n-m_i+1. $$
Using the residue theorem if follows that
\begin{equation}\label{ipad2*}
\gamma_{n,k}^{(i)}-A_{n,k}^{(i)}=\displaystyle \sum_{j=1}^{N}\displaystyle \mbox{Res}(Q_{n,{\bf m}}f_i\psi_k,z_j), \quad k\geq n-m^* +1,
\end{equation}
and
$$\mbox{Res}  (Q_{n,m}f_i\psi_k,z_j)=$$
 $$\frac{1}{(\tau_j-1)!}\lim_{z \to  z_j}\left[(Q_{n,{\bf m}}\psi_{n})(z)\frac{(z-z_j)^{\tau_j}f_i(z)\psi_k(z)}{\psi_{n}(z)}   \right]^{(\tau_j-1)}.$$

Using the Leibnitz formula, it follows that
\[
\left[(Q_{n,{\bf m}}\psi_{n})(z)\frac{(z-z_j)^{\tau_j}f_i(z)\psi_k(z)}{\psi_{n}(z)}   \right]^{(\tau_j-1)}=
\]
\begin{equation}\label{ipad4*}
\sum_{l=0}^{\tau_j-1} {\tau_j-1 \choose \ell} \left[ (Q_{n,{\bf m}}\psi_{n })(z) \right]^{(\tau_j -1-\ell)} \left[ \frac{(z-z_j)^{\tau_j}f_i(z)\psi_k(z)}{\psi_{n }(z)}\right]^{(\ell)}.
\end{equation}
Define
\begin{equation}\label{ipad5*}
\alpha_{n}(j,\ell):=\frac{1}{(\tau_j-1)!}{\tau_j -1 \choose \ell}\lim_{z \to z_j}\left[ Q_{n,m} \psi_{n }(z) \right]^{(\tau_j-1-\ell)}.
\end{equation}
These quantities do not depend on $i$ or $k$. By \eqref{ipad4*} and \eqref{ipad5*}, we obtain
\begin{equation} \label{ipad3*}
\mbox{Res}(Q_{n,m}f_i\psi_k,z_j)=\sum_{\ell=0}^{\tau_j-1}\alpha_{n}(j,\ell)\left[ \frac{(z-z_j)^{\tau_j}f_i(z)\psi_k(z)}{\psi_{n }(z)}\right]^{(\ell)}_{z=z_j}.
\end{equation}
Then, for each $k \geq n-m^* +1$ and $i=1,\ldots,d$, \eqref{ipad2*} and  \eqref{ipad3*} give
\begin{equation} \label{as} A_{n,k}^{(i)}=\gamma_{n,k}^{(i)}- \displaystyle \sum_{j=1}^{N} \sum_{\ell=0}^{\tau_j-1}\alpha_{n}(j,\ell)\left[ \frac{(z-z_j)^{\tau_j}f_i(z)\psi_k(z)}{\psi_{n }(z)}\right]^{(\ell)}_{z=z_j}.
 \end{equation}

For $k=n-m_i+1,\ldots,n, i=1,\ldots,d$
\begin{equation}\label{irod*}
\gamma_{n,k}^{(i)}=   \displaystyle \sum_{j=1}^{N} \sum_{\ell=0}^{\tau_j-1}\alpha_{n}(j,\ell)\left[ \frac{(z-z_j)^{\tau_j}f_i(z)\psi_k(z)}{\psi_{n }(z)}\right]^{(\ell)}_{z=z_j},
\end{equation}
since $A_{n,k}^{(i)}=0$ for these values of $k$.
Thus, we have obtained a system of $|\bf m|$ equations on $|\bf m|$ unknowns (the quantities $\alpha_n(j,\ell)$).

The determinant $\Delta_n$ of the system has the form
\[ \left|
\begin{smallmatrix}
\left[\frac{(z-z_j)^{\tau_j}f_i(z)\psi_{n- {m_i}+1}(z)}{\psi_{n}(z)}\right]_{z=z_j}  &  \cdots & \left[\frac{(z-z_j)^{\tau_j}f_i(z)\psi_{n- {m_i}+1}(z)}{\psi_{n}(z)}\right]_{z=z_j}^{(\tau_j-1)}  \\
\left[\frac{(z-z_j)^{\tau_j}f_i(z)\psi_{n- {m_i}+2}(z)}{\psi_{n}(z)}\right]_{z=z_j} &  \cdots & \left[\frac{(z-z_j)^{\tau_j}f_i(z)\psi_{n- {m_i}+2}(z)}{\psi_{n}(z)}\right]_{z=z_j}^{(\tau_j-1)} \\
\vdots & \vdots & \vdots \\
\left[\frac{(z-z_j)^{\tau_j}f_i(z)\psi_{n}(z)}{\psi_{n}(z)}\right]_{z=z_j} &  \cdots & \left[\frac{(z-z_j)^{\tau_j}f_i(z)\psi_{n}(z)}{\psi_{n}(z)}\right]_{z=z_j}^{(\tau_j-1)}
\end{smallmatrix}
\right|_{j=1,\ldots,N,i=1,\ldots,d} ,
\]
where the subindex on the determinant means that the indicated group of columns are successively written for $j=1,2,\ldots,N$ and the rows repeated for $i=1,\ldots,d$. Due to \eqref{Rak3}
$$ \lim_{n \longrightarrow \infty} \Delta_{n}= {\Delta},$$
where
$$ {\Delta}=
\left|\begin{smallmatrix}
[(z-z_j)^{\tau_j}z^{ {m_i}-1}f_i(z)]_{z=z_j} &   \cdots & [(z-z_j)^{\tau_j}z^{ {m_i}-1}f_i(z)]_{z=z_j}^{(\tau_j-1)}(z_j)  \\
[(z-z_j)^{\tau_j}z^{ {m_i}-2}f_i(z)]_{z=z_j}&  \cdots & [(z-z_j)^{\tau_j}z^{ {m_i}-2}f_i(z)]_{z=z_j}^{(\tau_j-1)}(z_j)\\
\vdots & \vdots &  \vdots \\
 [(z-z_j)^{\tau_j}f_i(z)]_{z=z_j}  &  \cdots & [(z-z_j)^{\tau_j}f_i(z)]_{z=z_j} ^{(\tau_j-1)}.
\end{smallmatrix}\right|_{j=1,2,\cdots,N, i=1,\ldots,d}.
 $$

Let us show that  $\Delta \neq 0$. Assume the contrary. Then there exists a linear combination of rows giving the zero vector. This means that there exist polynomials $p_1,\ldots,p_d, \deg p_i \leq m_i -1,$ such that
\[ \sum_{i=1}^d[(z-z_j)^{\tau_j}p_i(z)f_i(z)]^{(\ell)}_{z=z_j} =0, \quad j=1,\ldots,d, \quad \ell=0,\ldots,\tau_j-1,
\]
but this contradicts the assumption that $\bf f$ is polewise independent with respect to $\bf m$ in  $D_{|{\bf m}|}({\bf f})$. Consequently, $|\Delta_n| \geq C > 0$ for all sufficiently large $n$ and we restrict our attention to such $n$'s.

Using Cramer's rule, the system of equations \eqref{irod*} allows us to express the $\alpha_n(j,l)$ in terms of the $\gamma_{n,k}^{(i)}, i=1,\ldots,d, k=n- {m_i}+1,\ldots,n$. Arguing as in the proof of Theorem \ref{contenido} we arrive at the bounds (compare with \eqref{eq:as})
\begin{equation} \label{eq:As}
|A_{n,k}^{(i)}|\leq   \frac{C}{r^{k-n}(R - \delta)^n}, \quad k \geq n+1, \quad i=1,\ldots,d.
\end{equation}

Fix $i \in \{1,2,\ldots,d\}$. We have
\[ Q_{n,{\bf m}}(z) Q_{|\bf m|}(z) f_i(z)  - Q_{|\bf m|}(z)  {P}_{n,{\bf n},i}(z) = \sum_{\ell \geq n+1}  A^{(i)}_{n,\ell} Q_{|\bf m|}(z) \varphi_{\ell}(z)   =
\]
\begin{equation} \label{e5*} \sum_{\nu = 0}^{n+|{\bf m}| -m_i } B_{n,\nu}^{(i)} \varphi_{\nu}(z) + \sum_{\nu \geq n+|{\bf m}| -m_i +1}  B_{n,\nu}^{(i)} \varphi_{\nu}(z),
\end{equation}
where $Q_{|\bf m|}(z)  {P}_{n,{\bf n},i}(z)$ has degree at most $n+|{\bf m}| -m_i $.

Fix a compact set $K \subset D_{|\bf m|}(\bf f)$. In the sequel we assume that $R$ was chosen so that $T_R$ also surrounds $K$. For the series in \eqref{e5*}, as in the proof of Theorem \ref{contenido}, it is easy to show that
\[ \limsup_n \|\sum_{\nu \geq n+|{\bf m}| -m_i +1}  B_{n,\nu}^{(i)} \varphi_{\nu}(z)\|_K^{1/n} \leq \frac{\|z\|_K}{R_{|\bf m|}(\bf f)}.
\]
Here and below
$\|z\|_K$ is replaced by $1$ when $K \subset \overline{\mathbb{D}}$. In order to prove that
\[ \limsup_n \|\sum_{\nu =0}^{n+|{\bf m}| -m_i }  B_{n,\nu}^{(i)} \varphi_{\nu}(z)\|_K^{1/n} \leq \frac{\|z\|_K}{R_{|\bf m|}(\bf f)}
\]
one employs \eqref{eq:As} and the equality
\[ B_{n,\nu}^{(i)} = \sum_{\ell \geq n+1}  A^{(i)}_{n,\ell} \langle  Q_{|\bf m|}(z) \varphi_{\ell}(z),\varphi_\nu \rangle,
\]
in a similar fashion as in Theorem \ref{contenido}. Consequently, for each $i=1,\ldots,d$
\begin{equation} \label{e4}
\limsup_{n\to \infty} \|Q_{n,{\bf m}}  Q_{|\bf m|}  f_i  - Q_{|\bf m|}   {P}_{n,{\bf n},i} \|_{K}^{1/n} \leq \frac{\|z\|_K}{R_{|{\bf m}|}({\bf f})} .
\end{equation}
and
\begin{equation} \label{e4*} \limsup_{n\to \infty} \| f_i   -  {R}_{n,{\bf n},i}  \|_{K(\varepsilon)}^{1/n} \leq \frac{\|z\|_K}{R_{|{\bf m}|}({\bf f})}.
\end{equation}

From here convergence in Hausdorff content readily follows and we obtain that each pole of $\bf f$ in $D_{|\bf m|}(\bf f)$ attracts as many zeros of $Q_{n,\bf m}$ as its order. Since $\deg Q_{n,\bf m} \leq |\bf m|$ and the total number of poles of $\bf f$ in $D_{|\bf m|}(\bf f)$ equals $|\bf m|$ we have that $\deg Q_{n,\bf m} = |\bf m|$ for all sufficiently large $n$. This implies that ${\bf R}_{n,\bf m}$ is unique for such $n$'s. In fact, if this was not the case we could find an infinite subsequence of indices for which Definition \ref{defsimultaneosFP} has solutions with $\deg Q_{n,\bf m} < |\bf m|$, which contradicts what was proved.   In the sequel, we only consider such $n$'s. As the poles of $\bf f$ have absolute value greater than $1$, we obtain that
\[  Q_{n,{\bf m}}(z) =  \prod_{j=1}^{|\bf m|}(1 - \frac{z}{z_{n,j}})
\]
and
\[ \lim_n Q_{n,{\bf m}}(z) = \prod_{k=1}^{N} (1 - \frac{z}{z_{k}})^{\tau_k} = Q_{|\bf m|}(z).
\]
Since ${\mathcal{P}}_{|{\bf m}|}({\bf f})$ is the set of accumulation points of the zeros of $Q_{n,\bf m}$, \eqref{inequality3} follows at once from \eqref{e4*}.

Let us prove \eqref{inequality4}. To this end we start by proving that for $k=1,\ldots,N$
\begin{equation} \label{veloc}
\limsup_{n\to \infty} |Q_{n,{\bf m}}^{(j)}(z_k)|^{1/n} \leq |z_k|/R_{|{\bf m}|}({\bf f}),\qquad j=0,\ldots,\tau_k-1.
\end{equation}
Suppose that the pole $z_k$ attains its order with the function $f_i$. Let $\varepsilon >0$ be sufficiently small so that the closed disk $C_{k,\varepsilon} = \{z:|z-z_k| \leq \varepsilon\}$ is contained in $D_{|\bf m|}(\bf f)$ and contains no other pole of ${\bf f}$. On account of \eqref{e4}
\[
\limsup_{n\to \infty} \|(z-z_k)^{\tau_k} f_i Q_{n,{\bf m}}   -  (z-z_k)^{\tau_k} {P}_{n,{\bf m},i} \|_{C_{k,\varepsilon}}^{1/n} \leq \frac{\|z\|_{C_{k,\varepsilon}}}{R_{|{\bf m}|}({\bf f})},
\]
and using Cauchy's integral formula for the derivative, we have
\begin{equation} \label{eq:der}
\limsup_{n\to \infty} \|[(z-z_k)^{\tau_k} f_i Q_{n,{\bf m}}   -  (z-z_k)^{\tau_k} {P}_{n,{\bf m},i}]^{(j)} \|_{C_{k,\varepsilon}}^{1/n} \leq \frac{\|z\|_{C_{k,\varepsilon}}}{R_{|{\bf m}|}({\bf f})},
\end{equation}
for all $j\geq 0$. In particular, taking $z = z_k$ and $j=0$, we obtain
\[
\limsup_{n\to \infty} |AQ_{n,{\bf m}} (z_k)|^{1/n} \leq \frac{|z_k|}{R_{|{\bf m}|}({\bf f})},
\]
where $A= \lim_{z\to z_k} (z-z_k)^{\tau_k} f_i(z) \neq 0$ since $z_k$ is a pole of $f_i$ of order $\tau_k$. Therefore
\[\limsup_{n\to \infty} | Q_{n,{\bf m}} (z_k)|^{1/n} \leq \frac{|z_k|}{R_{|{\bf m}|}({\bf f})}.\]

Proceeding by induction, take $s \leq \tau_k$ and assume that
\begin{equation} \label{e5}
\limsup_{n\to \infty} |(Q_{n,{\bf m}}^{(j)}(z_k)|^{1/n} \leq \frac{|z_k|}{R_{|{\bf m}|}({\bf f})}, \qquad j=0,\ldots,s-2,
\end{equation}
and let us show that \eqref{e5} holds for $j=s-1$. As $s-1 < \tau$, using \eqref{eq:der} we deduce that
\begin{equation} \label{las} \limsup_{n\to \infty}  |[(z-z_k)^{\tau_k} f_i Q_{n,{\bf m}}]^{(s-1)}(z_k) |  \leq \frac{|z_k|}{R_{|{\bf m}|}({\bf f})},
\end{equation}
Applying, the Leibnitz formula it follows that
\[ [(z-z_k)^{\tau_k}f_iQ_{n,{\bf m}}]^{(s-1)}(z_k) = \sum_{\ell =0}^{s-1} {s -1 \choose \ell} Q_{n,{\bf m}}^{(s-1-\ell)}(z_k) [(z-z_k)^{\tau_k}f_i]^{(\ell)}(z_k).
\]
Using \eqref{e5}, \eqref{las}, and that $A \neq 0$, we conclude that
\[
\limsup_{n\to \infty} |(Q_{n,{\bf m}}^{(s-1)}(z_k)|^{1/n} \leq \frac{|z_k|}{R_{|{\bf m}|}({\bf f})}.
\]

Consider a basis of polynomials $\{q_{k,s}: k=1,\ldots,N, s=0,\ldots,\tau_k -1\}$ such that $\deg q_{k,s} \leq |{\bf m}| -1$ for all $k,  s$ and
\[ q_{k,s}^{(j)}(z_i) = \delta_{i,k}\delta_{j,s}, \quad 1\leq i \leq N,\quad 0\leq j \leq \tau_i -1.
\]
Then
\[ Q_{n,\bf m}(z) = \sum_{k=1}^{N} \sum_{s=0}^{\tau_k -1} Q_{n,\bf m}^{(s)}(z_k) q_{k,s}(z) + C_{n} Q_{|\bf m|}(z)
\]
where $C_n = \prod_{k=1}^N {z_k^{\tau_k}}/\prod_{j=1}^{|\bf m|} {z_{n,j}}$. From \eqref{veloc} it readily follows that
\[ \limsup_n \|Q_{n,\bf m}  - C_{n} Q_{|\bf m|} \|^{1/n} \leq \frac{\max\{|\zeta|: \zeta \in
{\mathcal{P}}_{|{\bf m}|}({\bf f})\}}{R_{|{\bf m}|}({\bf f})}.
\]
Evaluating at zero, we obtain
\[ \limsup_n |1- C_n|^{1/n} \leq \frac{\max\{|\zeta|: \zeta \in
{\mathcal{P}}_{|{\bf m}|}({\bf f})\}}{R_{|{\bf m}|}({\bf f})}
\]
which combined with the previous estimate gives us \eqref{inequality4}. We are done.  \hfill $\Box$

Corollary \ref{contenido} complements Theorem \ref{corfund} because under the assumptions of the latter it may still occur that $R_{m_i}(f_i) > R_{|\bf m|}(\bf f)$ for some $i$ so that \eqref{fund4*} gives a better estimate than \eqref{inequality3} for that particular $i$. It is also possible to construct examples where $f$ is not polewise independent with respect to $\bf m$ and using Corollary \ref{contenido} one can derive uniform convergence on compact subsets of the region obtained deleting from  $D_{m_i}(f_i)$ the poles of $f_i$ (see, examples in \cite[Section 5]{CCG}).

%%%%%%%%%%%%%%%%%%%%%%%%%%%%%%%%%%%%%%%%%%%%%%%%%%%%%%%%%%%%%%%%%%%%%%%%%%%%%%%%%%%%%%%%%%%%%%%%%%%%%%%%%%%%%%%%%%
%%%%%%%%%%%%%%%%%%%%%%%%%%%%%%%%%%%%%%%%%%%%%%%%%%%%%%%%%%%%%%%%%%%%%%%%%%%%%%%%%%%%%%%%%%%%%%%%%%%%%%%%%%%%%%%%%%
%%%%%%%%%%%%%%%%%%%%%%%%%%%%%%%%      REFERENCES    %%%%%%%%%%%%%%%%%%%%%%%%%%%%%%%%%%%%%%%%%%%%%
%%%%%%%%%%%%%%%%%%%%%%%%%%%%%%%%%%%%%%%%%%%%%%%%%%%%%%%%%%%%%%%%%%%%%%%%%%%%%%%%%%%%%%%%%%%%%%%%%%%%%%%%%%%%%%%%%%%
%%%%%%%%%%%%%%%%%%%%%%%%%%%%%%%%%%%%%%%%%%%%%%%%%%%%%%%%%%%%%%%%%%%%%%%%%%%%%%%%%%%%%%%%%%%%%%%%%%%%%%%%%%%%%%%%%%

%% References

\end{document}